\documentclass[10pt]{amsart}
\usepackage{graphicx}
\usepackage{hyperref}
\newtheorem{thm}{Theorem}[section]

\newtheorem{lem}[thm]{Lemma}
\newtheorem{prop}[thm]{Proposition}
\theoremstyle{definition}

\theoremstyle{remark}

\begin{document}

\title[A simple proof of Tong's theorem]{A simple proof of Tong's theorem}%
\author{Vyacheslav M. Abramov}%
\address{24 Sagan Drive, Cranbourne North, Victoria 3977, Australia}%

\email{vabramov126@gmail.com}%

\begin{abstract} We provide a new simple and transparent proof of the version of Kummer's test given in [Tong, J. (1994). \emph{Amer. Math. Monthly.} 101(5): 450--452]. Our proof is based on an application of a Hardy--Littlewood Tauberian theorem.\\

\end{abstract}
\maketitle

\section{Introduction.} Let
\begin{equation}\label{1}
\sum_{n=0}^{\infty}a_n
\end{equation}
be a positive series.  In 1835, Ernst Kummer \cite{K}  built a universal ratio test for convergence or divergence of positive series. The original version of the test was restrictive, and its proof was complicated. After more than fifty years since its publication Stoltz provided the clearer formulation and proof that has been well-accepted and appeared in the textbooks (see e. g. \cite[page 311]{Kn}) and well-known in our days. Being more general than many of existing particular tests such as d'Alembert test, Raabe's test, Bertrand's test and Gauss's test, it has been of the constant attention in the literature. It serves as a source of new particular tests (e.g. \cite{A}), has applications in the theory of difference equations (e. g. \cite{GH}), and there is a variety of new approaches to this test in the literature
(see \cite{D, Sa, S, T}).

A new version of Kummer's test was established by Tong \cite{T}. It was shown in \cite{T} that the modified version of Kummer's theorem covers
all positive series, and the convergence or divergence can be formulated in the form of two necessary and sufficient conditions, one for convergence and another for divergence. So, the result of Tong \cite{T} is a further important step of the development of Kummer's test. Some new proofs of Tong's result can be found in \cite{D, Sa, S}.

In the present note, we provide a simple proof of Tong's theorem \cite{T}. Our approach is based on an application of a known Hardy-Littlewood Tauberian theorem, and, being simple, the proof is mostly free of algebraic derivations.

\section{Tong's theorem and its proof.} We provide a new proof of the theorem given in \cite{T}. The formulation of that theorem is as follows.
\begin{thm} We have the following two claims.
\begin{enumerate}
\item [(i)] Series \eqref{1} converges if and only if there exists a positive sequence $\zeta_n$, $n=0,1,\dots$, such that $\zeta_n(a_n/a_{n+1})-\zeta_{n+1}\geq c>0$.
\item [(ii)] Series \eqref{1} diverges if and only if there exists a positive sequence $\zeta_n$, $n=0,1,\dots$, such that $\zeta_n(a_n/a_{n+1})-\zeta_{n+1}\leq0$ and $\sum_{n=0}^{\infty}1/\zeta_n=\infty$.
\end{enumerate}
\end{thm}

\begin{proof}
Recall the Hardy-Littlewood Tauberian theorem that will be used in this proof (see \cite{Hardy, HL}).
\begin{lem}\label{lem1} Let the series
$
\sum_{j=0}^{\infty}a_jx^j
$
converge for $-1<x<1$, and suppose that there exists $\gamma>0$ such that
$
\lim_{x\uparrow1}(1-x)^\gamma\sum_{n=0}^{\infty}a_jx^j=A.
$
Suppose also that $a_j\geq0$. Then, as $N\to\infty$, we have
$
\sum_{j=0}^Na_j=({A}/{\Gamma(1+\gamma)})N^\gamma(1+o(1)),
$
where $\Gamma(x)$ is Euler's Gamma-function.
\end{lem}
\subsection*{Proof of claim (i).}
Denoting $c_n=\zeta_n(a_n/a_{n+1})-\zeta_{n+1}$, write $c_na_{n+1}=z_n-z_{n+1}$, where $z_n=\zeta_na_n$. Now, for $|x|<1$ we use generating functions. Let $Z(x)=\sum_{n=0}^{\infty}z_nx^n$, and $A(x)=\sum_{n=0}^{\infty}c_na_{n+1}x^{n+1}$. We have the following relationship
\begin{equation}\label{2}
a_0\zeta_0-A(x)=(1-x)Z(x).
\end{equation}
\subsubsection*{Proof of `only if' part.}
If series \eqref{1} is convergent, then the sequence $\zeta_n$ can be chosen as follows. Suppose that the sequence $c_n$ is positive and bounded. Then Abel's theorem applies, and we have $\lim_{x\uparrow1}A(x)=\sum_{n=0}^{\infty}c_na_{n+1}=s$. So, $\zeta_0$ can be chosen such that $a_0\zeta_0-s>0$. Then, according to Lemma \ref{lem1}, for large $N$ we have $\sum_{n=0}^{N}z_n= (a_0\zeta_0-s)N(1+o(1))$, so the sequence of $\zeta_n$ exists. To this end, we are to show that the sequence $c_n$ is indeed bounded. Here we use an argument of \cite{S}. For some $n>n_0$ we have $z_{n+1}=z_n-c_na_{n+1}=z_n-(c_n-c+c)a_{n+1}$, where $c$ is the lower bound of the sequence $c_n$. So,
$z_{n+1}^\prime=z_{n+1}+(c_n-c)a_{n+1}=z_n-ca_{n+1}$. By this, we arrive at the new sequence $z_n^\prime$ with keeping all previous arguments the same. Hence, the relation $\zeta_n(a_n/a_{n+1})-\zeta_{n+1}\geq c>0$ implies the convergence of \eqref{1}.

\subsubsection*{Proof of `if' part.}
If we assume that the $\lim_{x\uparrow1}(1-x)Z(x)=C>0$ , we arrive at the claim that $\lim_{x\uparrow1}A(x)$ converges, since the left-hand side of \eqref{2} must converge to $C$ as $x\uparrow1$. Then the convergence of $\sum_{n=0}^\infty c_na_{n+1}$ follows from the inverse (Tauberian) theorem for positive series that in turn implies the convergence of \eqref{1}.

\subsection*{Proof of claim (ii).} We are to prove the following equivalent proposition.

\begin{prop}
Assume that $\zeta_n(a_n/a_{n+1})-\zeta_{n+1}\leq c\leq0$. Then series \eqref{1} converges if and only if $\sum_{n=0}^{\infty}1/\zeta_n<\infty$.
\end{prop}

\subsubsection*{Proof of `only if' part.} Assume that \eqref{1} converges. If $c_n$ is bounded, then using Lemma \ref{lem1} for large $N$ we arrive at the same estimate as in the proof of claim (i): $\sum_{n=0}^{N}z_n=(a_0\zeta_0-s)N(1+o(1))$. Taking into account that the sequence $z_n$ is increasing, it must converge to the limit. So, for large $n$ we have $\zeta_n=O(1/a_n)$, which means that the convergence of \eqref{1} implies the convergence of $\sum_{n=0}^{\infty}1/\zeta_n$. To this end, we are to show that the sequence $c_n$ can be assumed bounded. As in the proof of claim (i), for some $n>n_0$ we have $z_{n+1}=z_n-c_na_{n+1}=z_n-(c_n-c+c)a_{n+1}$, where $c$ is the upper bound of the sequence $c_n$. So,
$z_{n+1}^\prime=z_{n+1}+(c_n-c)a_{n+1}=z_n-ca_{n+1}$. By this, we arrive at the new sequence $z_n^\prime$ with keeping all previous arguments the same. Since $z_n^\prime=\zeta_n^\prime a_n\leq z_n$, we have $1/\zeta_n^\prime\geq1/\zeta_n$. Hence the convergence of $\sum_{n=0}^{\infty}1/\zeta_n^\prime$ implies the convergence of $\sum_{n=0}^{\infty}1/\zeta_n$.

\subsubsection*{Proof of `if' part.} If $\sum_{n=0}^{\infty}1/\zeta_n<\infty$, and $\zeta_{n+1}a_{n+1}-\zeta_na_n\geq-ca_{n+1}$ for large $n$, then $a_n\leq C/\zeta_n$ for some constant $C$, and \eqref{1} must converge.
\end{proof}


\begin{thebibliography}{2}

\bibitem{A}
Abramov, V. M. (2020). Extension of the Bertrand--De Morgan test
and its application. \emph{Amer. Math. Monthly}. 127(5): 444--448. \url{doi.org/10.1080/00029890.2020.1722551}

\bibitem{D}
\v{D}uri\v{s}, F. (2018). On Kummer's test of convergence and its relation to basic comparison tests. \url{arxiv.org/abs/1612.05167v2}

\bibitem{GH} Gy\H{o}ri, I., Horv\'{a}th, L. (2011). $l^p$-solutions and stability analysis of difference equations using the Kummer's test. \emph{Appl. Math. Comput.}, 217(24): 10129--10145. \url{doi.org/10.1016/j.amc.2011.05.008}

\bibitem{Hardy} Hardy, G. H. (2000). \emph{Divergent Series}, 2nd ed. Providence: AMS Chelsea Publishing.
\bibitem{HL} Hardy, G. H., Littlewood, J. E. (1914). Tauberian theorems concerning power series and Direchlet's series whose coefficients are positive. \emph{Proc. London Math. Soc.} 13: 174--191.     \url{doi.org/10.1112/plms/s2-13.1.174}


\bibitem{Kn}
Knopp, K. (1990). \textit{Theory and Application of Infinite Series}, New York: Dover Publ. Inc.  

\bibitem{K}
Kummer, E. (1835). \"Uber die Convergenz und Divergenz der unendlichen Reihen. \emph{J. f\"ur die Reine und Angew. Math.} 13: 171--184.


\bibitem{Sa}
Samelson, H. (1995). More on Kummer's test. \emph{Amer. Math. Monthly}. 102(9): 817–-818. \url{doi.org/10.1080/00029890.1995.12004667}

\bibitem{S}
Sj\"{o}din, T. (2018). A short and unified proof of Kummer's test. \url{arxiv.org/abs/1802.09858}

\bibitem{T}
Tong, J. (1994). Kummer's test gives a characterization of convergence or divergence for all positive series. \emph{Amer. Math. Monthly}. 101(5): 450--452. \url{doi.org/10.1080/00029890.1994.11996971}
\end{thebibliography}
\end{document}